\documentclass[psamsfont,a4paper]{amsart}
\pdfoutput=1

\usepackage{amsmath,amssymb,parskip}
\usepackage{hyperref}
\usepackage{bbm}
\usepackage{graphicx}
\usepackage[all]{xy}
\xyoption{rotate}

\newtheorem{theo}{Theorem}[section]

\newtheorem{prop}[theo]{Proposition}
\newtheorem{lemm}[theo]{Lemma}
\newtheorem{coro}[theo]{Corollary}
\newtheorem{conj}[theo]{Conjecture}

\newtheorem{rem}[theo]{Remark}

\newenvironment{rema}{\begin{rem}\rm }{\end{rem}}

\def\del{\partial} \def\eps{\varepsilon}
\def\bb#1{\mathbb{#1}} \def\m#1{\mathcal{#1}}

\def\id{\mathbbm{1}}
\def\im{\operatorname{im}}
\def\ra{\rightarrow} \def\lra{\longrightarrow} \def\co{\colon\thinspace}
\def\momeg{(M,\omega)}
\def\momegp{(M',\omega')}

\def\invg{\m O^{G}}

\def\invd{\m O^{\mathrm{Diff}_0(M)}}
\def\ham#1{\mathrm{Ham}#1}

\def\symp#1{\mathrm{Symp}#1}
\def\sympo#1{\mathrm{Symp}_0 #1}

\def\diffo{\mathrm{Diff}_0(M)}
\def\osc{\mathrm{osc}\,}
\def\flux{\mathrm{Flux}}
\def\ev{\mathrm{ev_1}}

\newcommand*{\quot}[2]%
{\ensuremath{%
   \raisebox{.35ex}{\ensuremath{#1}}\big/\raisebox{-.35ex}{\ensuremath{#2}}}}

\begin{document}

\title[Pseudo-distances on symplectomorphism groups and flux theory]{Pseudo-distances on symplectomorphism groups and applications to flux theory}
\author{Guy Buss, R\'emi Leclercq}
\date{\today}
\address{Guy Buss : Mathematisches Institut, Universit\"at Bonn, Endenicher Allee 60, D-53115 Bonn, Germany}
\address{R\'emi Leclercq : Centro de An\'alise Matem\'atica, Geometria e Sistemas Din\^amicos, Instituto Superior T\'ecnico, Av. Rovisco Pais, 1049-001 Lisboa, Portugal}
\email{guybuss@math.uni-bonn.de}
\email{leclercq@math.ist.utl.pt}

\begin{abstract}
   Starting from a given norm on the vector space of exact 1--forms of a compact symplectic manifold, we produce pseudo-distances on its symplectomorphism group by generalizing an idea due to Banyaga. We prove that in some cases (which include Banyaga's construction), their restriction to the Hamiltonian diffeomorphism group is equivalent to the distance induced by the initial norm on exact 1--forms. We also define genuine ``distances \textit{to} the Hamiltonian diffeomorphism group'' which we use to derive several consequences, mainly in terms of flux groups.
\end{abstract}

\maketitle

\section{Introduction}

Let $\momeg$ be a symplectic manifold. The non-degeneracy of $\omega$ induces an isomorphism between vector fields and $1$--forms on $M$ and the group of Hamiltonian diffeomorphisms of $\momeg$, denoted as usual $\ham{\momeg}$, consists of those diffeomorphisms which are induced by the flow of \textit{exact vector fields}, that is, vector fields corresponding to exact $1$--forms. In other words, $\ham{\momeg}$ is a Lie group whose Lie algebra is canonically identified with the vector space of exact $1$--forms of $M$ (or equivalently, the vector space of smooth functions up to constants, $C^\infty(M)/\bb R$).

It is naturally endowed with a well-known distance due to Hofer \cite{Hofer90}, and now called Hofer's distance, which is the Finsler distance induced by (the norm on the vector space of exact $1$--forms induced by) the $L^{\infty}$--norm on $C^\infty(M)$. The resulting geometry of the Hamiltonian diffeomorphism group has been widely studied and we will not try to exhaust here the different directions which have been explored. Let us only mention Polterovich's book \cite{Polterovich01} as an excellent introduction to the subject.

Recently, Banyaga ``extended'' Hofer's distance to the whole (connected component of the identity of the) symplectomorphism group \cite{Banyaga07}. This group consists of diffeomorphisms induced by closed (and not only exact) vector fields. The main idea behind Banyaga's construction is to pick a metric on the manifold and to (Hodge) decompose closed 1--forms as sums of harmonic and exact 1--forms. Then, exact 1--forms are dealt with as above by considering the $L^\infty$--norm, while any norm on the vector space of harmonic 1--forms does the trick (since this vector space is isomorphic to $H^1_{\mathrm{dR}}(M)$ and thus finite dimensional), see \S\ref{sec:example-banyaga} for a more detailed description of the construction.

In this paper, we generalize this idea by introducing an abstract notion of \emph{splitting seminorm} which formalizes Banyaga's construction in the most general way. Splitting seminorms lead to pseudo-distances on the symplectomorphism group which, in turn, lead to genuine distances \textit{to} the Hamiltonian diffeomorphism group. Finally, we derive from the latter consequences (mainly) in terms of flux groups.

\subsection*{The general abstract construction}

Let $\momeg$ be a closed symplectic manifold. We denote respectively by $Z^1(M)$ and $B^1(M)$ the vector spaces of closed and exact 1--forms of $M$. We fix a norm $n^B$ on $B^1(M)$, and a metric $g$ on $M$. We denote by $\|-\|_{L^2}$ the $L^2$--norm induced by $g$ on 1--forms. 

We emphasize the fact that the results we present here \emph{do not depend} on the choice of the metric $g$, since on closed manifolds $L^2$--norms induced by different metrics are equivalent. 

First, we define the notion of \emph{splitting (semi)norms} for $n^B$, that is, seminorms on $Z^1(M)$ of the form 
\begin{align*}
       n_{(\mu,c)}(\alpha) = n^B(\mu(\alpha)) + c \|\alpha - \mu(\alpha)\|_{L^2}  
\end{align*}
where $c \geq 0$ is a real number and $\mu$ is a homomorphism $\mu\co Z^1(M)\ra B^1(M)$.

We denote by $\m P X$ (respectively $\m P_x X$), the space of paths in $X$ (respectively paths in $X$, starting at $x$). Recall that we have the following one-to-one correspondences
\begin{align}\label{eq:correspondences-paths}
  \begin{split}
    \xymatrix@R=0.4cm@C=0.1cm{\relax
       \m P_{\id} \symp{(M,\omega)} \ar@{<->}[d] & \ni & \Psi=\{\psi_t\}_{t\in [0,1]} & \\
       \m P \Gamma_{\mathrm{Symp}}(TM) \ar@{<->}[d] & \ni & X^\Psi=\{X^\Psi_t\}_{t\in [0,1]} \mbox{ such that } \del_t(\psi_t) = X^\Psi_t(\psi_t)  \\
       \m PZ^1(M) & \ni & \alpha(\Psi)=\{\alpha(\Psi)_t\}_{t\in [0,1]} \mbox{ such that } \alpha(\Psi)_t = \omega(X^\Psi_t,\cdot\,)\\
  }
  \end{split}
\end{align}
and that under this correspondence, $\m P_\id \ham{\momeg}$ corresponds to $\m PB^1(M)$.

Thus $n_{(\mu,c)}$ naturally leads to a (pseudo-)length $\ell_{(\mu,c)}$ on $\m P_{\id} \symp{(M,\omega)}$:
$$ \ell_{(\mu,c)}(\Psi) =  \int_0^1 \Big( n^B(\mu(\alpha(\Psi)_t)) + c \| \alpha(\Psi)_t - \mu (\alpha (\Psi)_t)\|_{L^2} \Big) dt $$
which, in turn, leads to a pseudo-distance $d^s_{(\mu,c)}$ given by
$$ d_{(\mu,c)}(\id, \phi) = \inf_{\Phi \mid \ev(\Phi) = \phi}\ell_{(\mu,c)} (\Phi)\;, \qquad d_{(\mu,c)} (\phi, \psi) = d_{(\mu,c)}(\id, \phi^{-1}\psi)$$
up to a necessary symmetrization:
$$ d^s_{(\mu,c)}(\phi, \psi)= \frac 12 \left(d_{(\mu,c)}(\phi, \psi) + d_{(\mu,c)}(\psi,\phi)\right)$$
(since $d_{(\mu,c)}$ is not symmetric in general).

In what follows, we prove the following facts:
\begin{itemize}
\item[(1)] $d^s_{(\mu,c)}$ is indeed a pseudo-distance on $\sympo{\momeg}$ (Proposition \ref{lemm:prop-of-d-and-ds}).
\item[(2)] 
  \begin{itemize}
     \item [(a)] On $\ham{\momeg}$, the restriction of $d^s_{(\mu,c)}$ is bounded above by the distance induced by the restriction of $n_{(\mu,c)}$ (Proposition~\ref{prop:comparison-with-nb}).
     \item [(b)] In some cases, these two distances are even equivalent (Theorem \ref{theo:equiv-bany-hofers}). 
  \end{itemize}
\item[(3)] More surprisingly, as soon as $c$ is non-zero, the possible degeneracy of $d^s_{(\mu,c)}$ lies in $\ham{\momeg}$ (Proposition~\ref{prop:non-degeneracy}).
\end{itemize}

Concerning point (2), there are two natural ``restrictions'' of our construction to $\ham{\momeg}$: The actual \textit{restriction of the distance} $d^s_{(\mu,c)}$ to the subgroup $\ham{\momeg}$ and the distance on $\ham{\momeg}$ induced by \textit{the restriction of the norm} $n_{(\mu,c)}$ to $B^1(M)$. The two resulting distances are not equivalent in general and (2a) states that the former is bounded above by the latter while (2b) states that there are non-trivial cases where they are equivalent. 

Recall that (as claimed above), the pseudo-distances $d^s_{(\mu,c)}$'s generalize Banyaga's construction (we precisely state this fact later). In particular, the cases appearing in (2b) include Banyaga's distance and thus we answer positively a question he raised in \cite{Banyaga07}. (Actually and as we shall see, the proof of Theorem \ref{theo:equiv-bany-hofers} -- as well as the proof of Proposition~\ref{prop:non-degeneracy} -- heavily relies on ideas used by Banyaga in \cite{Banyaga07}.)

Moreover, the non-degeneracy of Banyaga's distance indicates that in some cases $d^s_{(\mu,c)}$ is a genuine distance. We do not investigate their non-degeneracy in full generality, however, in view of (3) above, we are able to derive non-degenerate distances \textit{to} $\ham{\momeg}$ as follows.
\begin{prop}\label{coro:dist2Ham}
  When $c$ is non-zero, $d_{(\mu,c)}$ induces a \emph{non-degenerate} distance to $\ham{\momeg}$, $\Delta_{(\mu,c)}\co \sympo{\momeg}\ra \bb R^+$, by the formula
  \begin{align*}
    \Delta_{(\mu,c)} (\psi) = \inf \{ d_{(\mu,c)}(\varphi,\psi)  \,|\;\varphi\in\ham{\momeg} \} \; .
  \end{align*}
This distance is bi-invariant with respect to the action of $\ham{\momeg}$ by composition and thus induces a map on the quotient 
\begin{align*}
\delta_{(\mu,c)}\co \sympo{\momeg}/\ham{\momeg} \lra \bb R^+  
\end{align*}
which vanishes only at $[\id]$. Finally, $\Delta_{(\mu,c)}$ satisfies 
\begin{align*}
\Delta_{(\mu,c)}(\psi'\psi) = \Delta_{(\mu,c)}(\psi\psi')\leq \Delta_{(\mu,c)}(\psi)+\Delta_{(\mu,c)}(\psi')  
\end{align*}
for all symplectomorphisms $\psi$ and $\psi'$ and thus, so does $\delta_{(\mu,c)}$ on equivalence classes.
\end{prop}

\subsection*{Applications}

We use these $\Delta_{(\mu,c)}$'s to derive several easy consequences in terms of (non-) Hamiltonian symplectomorphisms (Corollary \ref{coro:hamilt-square-sympl}) and natural mappings of the symplectomorphism group (Corollary~\ref{coro:conjug-pres-shperes} and Proposition~\ref{prop:bounded-mult}). However the main applications of our machinery are expressed in terms of Flux groups, which we now turn to describe.

\subsubsection*{Lattices of $H^1_\mathrm{dR}(M)$}

The flux group of a compact symplectic manifold $(M,\omega)$ is denoted $\Gamma_M$ and defined as the image of $\pi_1(\sympo{\momeg})$ via the flux morphism:
\begin{align}\label{eq:def-flux-morph}
  \flux\co \widetilde{\mathrm{Symp}_0}\momeg \lra H^1_{\mathrm{dR}}(M) \;, \qquad [\Psi] \longmapsto \int_0^1 [\alpha(\Psi)_t]\, dt \;.
\end{align}
By Ono's theorem \cite{Ono06}, $\Gamma_M$ is known to be a discrete subgroup of $H^1_{\mathrm{dR}}(M)$ and a question we are interested in is: How far from 0 does the first non-trivial element of $\Gamma_M$ lie ? (And: What does ``how far'' mean ?) Even though we are not able to provide a satisfying answer to this question, our $\Delta$'s lead to obvious, non-unrelated observations as the following one.
\begin{prop}\label{prop:flux-group-lattice}
  If $\Delta(\sympo{\momeg})$ is unbounded (in $\bb R^+$), the flux group is not a lattice of $H^1_{\mathrm{dR}}(M)$, that is,
  \begin{align*}
    \quot{H^1_{\mathrm{dR}}(M)}{\Gamma_M} \simeq \bb R^k\times \bb T^{b_1(M)-k}  \qquad \mbox{ with }\; k\neq 0
 \end{align*}
($b_1(M)$ denotes the first Betty number of $M$ and $\bb T^m$ the $m$--dimensional torus).
\end{prop}
Notice that $\Delta$ stands for $\Delta_{(\mu,c)}$ for any choice of data $(\mu,c)$, provided that $c\neq 0$.

\subsubsection*{Flux group of Cartesian products}

The question whether the flux group of a product of symplectic manifolds is isomorphic to the product of their flux groups has useful consequences. (See for example the \emph{Bounded isometry conjecture} by Lalonde and Polterovich \cite{LalondePolterovich97} for products of surfaces of positive genus and recent extensions by Campos-Apanco and Pedroza \cite{CamposPedroza10}.)

More precisely, let $(M,\omega)$ and $(M',\omega')$ be compact symplectic manifolds. It is easy to see that $\Gamma_M\times\Gamma_{M'}$ is ``included'' in $\Gamma_{M\times M'}$ and that they are isomorphic if and only if all split Hamiltonian diffeomorphisms of the product, $\phi\times\psi$ (with $\phi$ and $\psi$ symplectomorphisms of $(M,\omega)$ and $(M',\omega')$ respectively), are products of Hamiltonian diffeomorphisms. (See \S \ref{sec:appl-flux-groups} for more details and proofs.)

In other words, the obstruction for $\Gamma_{M\times M'}\simeq \Gamma_M\times\Gamma_{M'}$ is the existence of non-Hamiltonian symplectomorphisms of each component whose product is Hamiltonian. Our machinery gives information about those. 
\begin{theo}\label{theo:appli-flux-groups}
  Let $(M,\omega)$ and $(M',\omega')$ be closed, connected, symplectic manifolds. Let $\phi \in\sympo{(M,\omega)}\backslash \ham{(M,\omega)}$. There exists $\eps=\eps(\phi)>0$ such that
  \begin{align*}
    \phi\times \psi \in \ham{(M\times M',\omega\oplus\omega')} \quad \Longrightarrow \quad \Delta(\psi) \geq \eps \; .
  \end{align*}
In other words, given a non-Hamiltonian symplectomorphism of $(M,\omega)$, $\phi$, there exists a $\Delta$--neighborhood of $\ham{(M',\omega')}$, $\m U$, such that the map
\begin{align*}
  \sympo{(M',\omega')} \longrightarrow \sympo{(M\times M',\omega\oplus\omega')} \;, \qquad \psi \longmapsto \phi\times \psi
\end{align*}
restricted to $\m U$ takes its values in $\sympo{(M\times M',\omega\oplus\omega')} \backslash \ham{(M\times M',\omega\oplus\omega')}$.
\end{theo}
(Here also, $\Delta$ stands for $\Delta_{(\mu,c)}$ for any choice of data $(\mu,c)$, provided that $c\neq 0$. Given $\phi$, the constant $\eps(\phi)$ itself depends on this choice, however it is non-zero for any choice.)

Curiously enough, the proof of this theorem boils down to the fact that if $\phi\times\psi$ is Hamiltonian, then necessarily, $\Delta(\id_M\times \psi)=\Delta(\phi^{-1}\times \id_{M'}) \, (=\eps(\phi))$.

\subsection*{More natural $\mu$'s (do not exist)}

We also discuss a special case of morphisms $\mu$ satisfying a quite technical condition which appears at several places and makes the theory much nicer (and to which we refer as Condition $(\dagger)$). 

Then we conjecture that no homomorphism $\mu$ can satisfy this condition (Conjecture~\ref{conj:nonexist-dagger-mu}) and we provide evidence towards this fact in terms of representation theory of the diffeomorphism group on closed 1--forms. (In particular, we prove, as a bi-product, that the pullback action of $\diffo$ on $Z^1(M)$ is irreducible, even though $Z^1(M)$ is not simple, see Corollary~\ref{theo:non-trivial-intersection}.)

\subsection*{Organization of the paper} 

The rest of the paper follows the order of the introduction. In Section~\ref{sec:general-case}, we deal with the general abstract construction: We precisely state and prove the properties of $d_{(\mu,c)}$ mentioned above. In Section \ref{sec:applications}, we state and prove some applications (quickly mentioned above). We also provide more details and we prove Theorem \ref{theo:appli-flux-groups}. In Section \ref{sec:ConditionDagger}, we gather our thoughts concerning Condition $(\dagger)$. Finally, in Section~\ref{sec:examples}, we present several examples of our construction, including the case studied by Banyaga.

\subsection*{Acknowledgments} 

This project started when both authors were at the Max Planck Institute of Leipzig. However, its geometry has drastically changed over the months and  it took its final shape during two stays of the second author at the University of Bonn. Both authors would like to acknowledge these extremely good working environments. The second author would also like to associate the Instituto Superior T\'ecnico of Lisbon to the previous comment. 

The authors also thank Augustin Banyaga, for disambiguation concerning concatenation of paths and more importantly for \cite{Banyaga07} which motivated this work. Finally, they thank the referee for valuable comments and for pointing out \cite{BuhovskyOstrover10} to them. 

The final publication is available at \href{http://www.springerlink.com/content/m47k168280460m32/}{www.springerlink.com}.

\section{The general abstract construction}\label{sec:general-case}

In what follows, $(M,\omega)$ is a closed symplectic manifold. We fix a Riemannian metric $g$ on $M$ and a norm $n^B$ on $B^1(M)$.

\subsection{Pseudo-distances on $\sympo{\momeg}$}

We recall that the starting point of our construction is the notion of \emph{splitting (semi)norms} on $Z^1(M)$:
   $$ n_{(\mu,c)}(\alpha) = n^B(\mu(\alpha)) + c \|\alpha - \mu(\alpha)\|_{L^2} $$
where $\mu\co Z^1(M) \rightarrow B^1(M)$ is a linear map and $c$ a non-negative real number.

Notice that \textit{any seminorm} on $Z^1(M)$ which induces $n^B$ on $B^1(M)$ is equivalent to such a seminorm since a complementary space of $B^1(M)$ has finite dimension.

\begin{lemm}\label{lemm:criterion-SSN-norms}
   $n_{(\mu,c)}$ is a norm on $Z^1(M)$ if and only if either $c \neq 0$ or $\mu$ is injective.
\end{lemm}

\begin{proof}
The only property we need to check is non-degeneracy. Clearly, when $c\neq 0$, if $n_{(\mu,c)}(\alpha)=0$, then $n^B(\mu(\alpha))$ and $\| \alpha - \mu(\alpha) \|_{L^2}$ vanish. Since $n^B$ is non-degenerate, we immediately get that $\| \alpha \|_{L^2}=0$ and thus $n_{(\mu,c)}$ is non-degenerate. \\
Now if $c=0$, $n_{(\mu,c)}(\alpha)=0$ is equivalent to $n^B(\mu(\alpha))=0$, that is $\mu(\alpha)=0$, and $n_{(\mu,c)}$ is non-degenerate if and only if $\mu$ is injective. 
\end{proof}

Let us remark several things about the definition:
\begin{enumerate}
	\item{$c$ acts as a switch: Its value is not of importance provided that $c\neq 0$.}
	\item{When $\mu=0$,  $n_{(\mu,c)}$ is the $L^2$--norm. On the other hand, choosing $n^B$ as the $L^2$--norm restricted to $B^1(M)$ leads to $L^2$--norms weighted by $\mu$.}
	\item{If $\mu$ is a projection onto a subspace $A\subset B^1(M)$ and $c=1$,  the restriction of $n_{(\mu,c)}$ to $B^1(M)$ is a genuine norm which \emph{interpolates} between the $L^2$--norm and the norm induced on $A$ by $n^B$. Moreover, it coincides with the $L^2$--norm for $A=\{0\}$ and with $n^B$ for $A=B^1(M)$.}
\item The latter case ($\mu$ projection on $B^1(M)$, even with $c=0$) is interesting since then $n_{(\mu,c)}$ is a genuine extension of $n^B$ to a (semi)norm on $Z^1(M)$.
\end{enumerate}

\begin{rema}\label{rema:injective-mu}
Since we are interested in extending distances on $\ham{\momeg}$ to (pseudo-)distances on $\sympo{\momeg}$, we will not pay much attention to injective $\mu$'s. Indeed, even though there exist such injective linear operators when the first de Rham cohomology group $H^1_{\mathrm{dR}}(M)$ does not vanish, they are particularly unnatural for they necessarily act non-trivially on an infinite dimensional subspace of $B^1(M)$.
\end{rema}

We recall from \eqref{eq:correspondences-paths}, that $ \Psi=\{\psi_t\}_{t\in [0,1]}\in \m P_{\id} \symp{(M,\omega)}$ uniquely corresponds to $ \alpha(\Psi)=\{\alpha(\Psi)_t\}_{t\in [0,1]}\in\m PZ^1(M)$, where $\alpha(\Psi)_t=\omega(\del_t \psi_t(\psi_t^{-1}),\cdot\,)$. 
We also recall from the introduction that any splitting seminorm $n_{(\mu,c)}$ induces a (pseudo-)length $\ell_{(\mu,c)}$ on $\m P_{\id} \symp{(M,\omega)}$:
$$ \ell_{(\mu,c)}(\Psi) =  \int_0^1 \Big( n^B(\mu(\alpha(\Psi)_t)) + c \| \alpha(\Psi)_t - \mu (\alpha (\Psi)_t)\|_{L^2} \Big) dt $$
which, in turn, leads to
\begin{align*}
   & d_{(\mu,c)}(\id, \phi) = \inf_{\Phi \mid \ev(\Phi) = \phi} \ell_{(\mu,c)} (\Phi)\;, \qquad d_{(\mu,c)} (\phi, \psi) = d_{(\mu,c)}(\id, \phi^{-1}\psi)  \\
   & \mbox{and }\qquad d^s_{(\mu,c)}(\id, \phi)= \frac 12 \left(d_{(\mu,c)}(\phi,\psi) + d_{(\mu,c)}(\psi,\phi)\right)\;.
\end{align*}

These quantities satisfy the following properties.
\begin{prop}
\label{lemm:prop-of-d-and-ds}
  Let $\mu\co Z^1(M)\ra B^1(M)$ be any homomorphism and $c\geq 0$.
  \begin{enumerate}
  	\item\label{list:ps-dist} $d_{(\mu,c)}$ is positive and satisfies the triangle inequality such that $d_{(\mu,c)}^s$ is a pseudo-distance on $\sympo{\momeg}$.
  	\item\label{list:invariance} $d_{(\mu,c)}$ is left-invariant but not right-invariant in general.
  	\item\label{list:Finsler} If $c=0$ and $\mu$ satisfies $n^B(\mu(\alpha))=n^B(\mu(\phi^*\alpha))  $ for all $\alpha\in Z^1(M)$ and all $\phi\in\symp_0{\momeg}$, then $d_{(\mu,0)}$ is symmetric (and $d_{(\mu,0)}=d_{(\mu,0)}^s$).
 	\end{enumerate}
\end{prop}

\begin{rema}
\label{rema:ConditionDagger}
  We will refer to the condition of assertion (\ref{list:Finsler}) as condition $(\dagger)$:
          \begin{align*}
            \forall\alpha\in Z^1(M), \, \forall\phi\in\symp_0{\momeg}, \quad n^B(\mu(\alpha))=n^B(\mu(\phi^*\alpha))  \;.
            \tag{$\dagger$}
          \end{align*}
Notice that it is equivalent to requiring that $n^B(\mu(\alpha))\leq n^B(\mu(\phi^*\alpha))$ for any $1$--form $\alpha$ and any symplectomorphism $\phi$. Since 
\begin{align*}
  	 \alpha(\Phi^{-1} \circ \Psi) = \alpha(\Phi^{-1}) + \Phi^*\alpha(\Psi) =-\Phi^* (\alpha(\Phi) - \alpha(\Psi))
\end{align*}
(where $\circ$ denotes time-wise composition), when $c=0$, condition $(\dagger)$ leads to the equality 
\begin{align*}
   \ell_{(\mu,c)}(\alpha(\Phi^{-1} \circ \Psi)) = \ell_{(\mu,c)}(\alpha(\Phi) - \alpha(\Psi))  
\end{align*}
which relates group structure on $\m P_\id\symp{\momeg}$ and linear structure on $\m PZ^1(M)$. This has useful consequences (see Section~\ref{sec:ConditionDagger} for a more detailed discussion). 
\end{rema}

\begin{proof}[Proof of Proposition \ref{lemm:prop-of-d-and-ds}]
 First, note that (\ref{list:invariance}) is obvious by definition. 

The proof of (\ref{list:Finsler}) is also straightforward: When $c=0$, condition $(\dagger)$ leads to the equality $\ell_{(\mu,c)}(\Phi^{-1}) =\ell_{(\mu,c)}(\Phi)$ (since $\alpha(\Phi^{-1})=-\Phi^*\alpha(\Phi)$) and $d_{(\mu,0)}$ is symmetric. 

Now, the triangle inequality (\ref{list:ps-dist}) satisfied by $d_{(\mu,c)}$ easily follows from an appropriate choice of ``product'' of paths (see \S~\ref{subsection:timewise-composition-and-concatenations} below for related remarks). Indeed, pick any three symplectomorphisms isotopic to $\id$, say $\varphi_1$, $\varphi_2$, and $\varphi_3$. We want to prove that
\begin{align}
  \label{eq:triangle-ineq-d}
  d_{(\mu,c)}(\varphi_1, \varphi_3) \leq d_{(\mu,c)}(\varphi_1, \varphi_2) + d_{(\mu,c)}(\varphi_2, \varphi_3) \; .  
\end{align}
Now, by definition,
$$ d_{(\mu,c)}(\varphi_1, \varphi_3) = d_{(\mu,c)}(\id,\varphi_1^{-1} \varphi_3) = \inf \ell_{(\mu,c)} (\Xi) $$
where $\Xi$ is any smooth symplectic isotopy from $\id$ to $\varphi_1^{-1}\varphi_3$. Put $\phi=\varphi_1^{-1}\varphi_2$ and $\psi=\varphi_2^{-1}\varphi_3$ and choose any paths $\Phi$ and $\Psi$ connecting $\id$ to $\phi$ and $\psi$ respectively.

Define $\Phi\ast_l \Psi$ as the ``left concatenation'' of the paths $\Phi$ and $\Psi$ as follows:
\begin{align}
  \label{eq:definition-concatenation-left}
  \Phi\ast_l\Psi\co \left\{
    \begin{array}[]{ll}
      \psi_{r(t)}, & 0\leq t\leq 1/2\\
      \phi_{s(t)}\psi, &1/2\leq t\leq 1
    \end{array}
  \right. 
\end{align}
where $r$ and $s$ are smooth, surjective, non-decreasing functions, defined on $[0,1/2]$ and $[1/2,1]$ respectively, with values in $[0,1]$ which are constant near the ends of their interval of definition. 

It is a smooth symplectic isotopy from $\id$ to $\phi\psi=(\varphi_1^{-1}\varphi_2)(\varphi_2^{-1}\varphi_3)=\varphi_1^{-1}\varphi_3$, and its corresponding $1$--form is easily seen to be 
\begin{align*}
  \alpha(\Phi\ast_l\Psi) &\co \!\!\left\{\!\!
    \begin{array}[]{ll}
      r'(t)\,\alpha(\Psi)_{r(t)}, & 0\leq t\leq 1/2\\
      s'(t)\,\alpha(\Phi)_{s(t)}, &1/2\leq t\leq 1
    \end{array}
  \right. 
\end{align*}
Thus, we get, after a harmless change of variables (in time):
\begin{align}\label{eq:equality-length-concatenation}
  \ell_{(\mu,c)} (\Phi\ast_l\Psi) =   \ell_{(\mu,c)} (\Phi) + \ell_{(\mu,c)} (\Psi) 
\end{align}
which, in turn, leads to
$$ d_{(\mu,c)}(\varphi_1, \varphi_3) = \inf \ell_{(\mu,c)} (\Xi) \leq \ell_{(\mu,c)} (\Phi\ast_l\Psi) =  \ell_{(\mu,c)} (\Phi) + \ell_{(\mu,c)} (\Psi) \; .  $$
Since this holds for all paths $\Phi$ and $\Psi$ (with prescribed end points), we deduce that
 $$ d_{(\mu,c)}(\varphi_1, \varphi_3) \leq \inf  \ell_{(\mu,c)} (\Phi) + \inf \ell_{(\mu,c)} (\Psi) =  d_{(\mu,c)} (\id,\phi) + d_{(\mu,c)} (\id,\psi)  $$
by independently taking the infima. This, in turn, leads us to \eqref{eq:triangle-ineq-d}. Now the positivity is clear and $d_{(\mu,c)}^s$ is symmetric by definition. Proposition \ref{lemm:prop-of-d-and-ds} is proved.
\end{proof}

\subsubsection{Remarks on products of paths}\label{subsection:timewise-composition-and-concatenations}

It is well-known that in terms of fundamental groups, time-wise composition and concatenation of loops induce the same group law. However, on loops (and not homotopy classes) time-wise composition is a group law, but concatenation only a monoid law. 

Moreover, there is no canonical way to define the concatenation of paths starting at the identity and different choices lead to different results in terms of length. Let $\Phi$ and $\Psi$ be paths of symplectomorphisms starting at $\id$ and respectively ending at $\phi=\phi_1$ and $\psi=\psi_1$. Similarly to \eqref{eq:definition-concatenation-left} above, we define their ``right concatenation''
\begin{align*}
  \Phi \ast_r \Psi\co \left\{
    \begin{array}[]{ll}
      \phi_{r(t)}, & 0\leq t\leq 1/2\\
      \phi\psi_{s(t)}, &1/2\leq t\leq 1
    \end{array}
  \right.
\end{align*}
with $r$ and $s$ as in \eqref{eq:definition-concatenation-left}. This path connects $\id$ to $\phi\psi$ and corresponds to the 1--form:
\begin{align*}
  \alpha(\Phi \ast_r \Psi) &\co \!\!\left\{\!\!
    \begin{array}[]{ll}
      r'(t)\,\alpha(\Phi)_{r(t)}, & 0\leq t\leq 1/2\\
      s'(t)\,(\phi^{-1})^*\alpha(\Psi)_{s(t)}, &1/2\leq t\leq 1
    \end{array}
  \right.
\end{align*}

In view of \eqref{eq:equality-length-concatenation}, we see that $\ast_l$ is more relevant to us here since in general
\begin{align}\label{eq:equality-length-concatenation-right}
  \ell_{(\mu,c)} (\Phi\ast_r\Psi) = \ell_{(\mu,c)} (\Phi) + \ell_{(\mu,c)} (\phi_1\Psi) \neq   \ell_{(\mu,c)} (\Phi) + \ell_{(\mu,c)} (\Psi)   \;.
\end{align}
See \S \ref{sec:difference-concatenations-Banyagas-case} and in particular Lemma \ref{lemm:concatenation-not-minimiz} for some more precise remarks in the context of Banyaga's construction. 

Notice also that there is no constant $C>0$ such that $\ell_{(\mu,c)}(\Phi\ast\Psi) \leq C\cdot \ell_{(\mu,c)}(\Phi\circ\Psi)$ ($\ast$ standing for $\ast_r$ or $\ast_l$ and $\circ$ denoting time-wise composition) since $\ell_{(\mu,c)}(\Phi\circ\Phi^{-1})=0$. Moreover, under condition ($\dagger$), we have for $c=0$:
$$\ell_{(\mu,0)}(\Phi\circ\Psi) \leq \ell_{(\mu,0)}(\Phi\ast_l\Psi) = \ell_{(\mu,0)}(\Phi\ast_r\Psi) \;. $$

\subsection{Comparison with the initial norm}

Now we can also ask the question of the equivalence between $d_{(\mu,c)}$ restricted to $\ham{\momeg}$ and the distances induced by the initial norm. Let us denote by $d^B_{(\mu,c)}$ the distance on $\ham(M,\omega)$ induced by $\ell_{(\mu,c)}$ by taking the infimum over the set of paths included in $\ham{\momeg}$.

\begin{prop}
\label{prop:comparison-with-nb}
In general, for all $\varphi\in \ham(M,\omega)$, 
\begin{align*}
\xymatrix@R=3pt@C=.5pt{\relax
    d_{(\mu,0)}(\id,\varphi)\;\; \ar@{}[rr]|*+<2pt>{\rotatebox{0}{$\leq$}}\ar@{}[rd]|*+<2pt>{\rotatebox{-37}{$\leq$}}  && d_{(\mu,c)}(\id,\varphi) \;\; \ar@{}[rd]|*+<2pt>{\rotatebox{-37}{$\leq$}} & \\
    & \;\;d_{(\mu,0)}^B(\id,\varphi) \ar@{}[rr]|*+<2pt>{\rotatebox{0}{$\leq$}} && \;\;d_{(\mu,c)}^B(\id,\varphi) \\
}
\end{align*}
\end{prop}

  Even though we cannot prove that $d_{(\mu,c)}$ and $d_{(\mu,c)}^B$ are equivalent on $\ham{\momeg}$ in this generality, we can do it in certain cases and in particular in the case studied by Banyaga in \cite{Banyaga07} (where he raised the question). See \S\ref{sec:equiv-bany-hofers-dist} for the precise result.

  \begin{proof}
This proposition is straightforward. The inequalities on the lines come from the fact that the length function $\ell_{(\mu,c)}$ is always greater or equal than $\ell_{(\mu,0)}$. The comparisons between the two lines come from the fact that we take the infima on a subspace when we consider the distances $d_{(\mu,c)}^B$ (for all $c$'s).    
  \end{proof}

\subsection{Distances to $\ham{\momeg}$}

We now precise the possible degeneracy of $d_{(\mu,c)}$.
\begin{prop}
\label{prop:non-degeneracy} 
Let us consider non-injective $\mu$'s. If $c= 0$, $d_{(\mu,0)}$ is degenerate. For $c\neq 0$, the degeneracy of $d_{(\mu,c)}$ is Hamiltonian, that is 
$$d_{(\mu,c)}(\id,\varphi)=0 \;\Longrightarrow\; \varphi\in\ham{\momeg} \; .$$
Equivalently, $d_{(\mu,c)}(\phi,\psi)=0$ forces $\psi = \phi\circ\varphi$ for some $\varphi\in\ham{\momeg}$.
\end{prop}

As a corollary of (the proof of) this proposition, we will see that the ``distances to $\ham{\momeg}$'', $\Delta_{(\mu,c)}$, defined in Proposition \ref{coro:dist2Ham}, are non-degenerate. (We prove Proposition \ref{coro:dist2Ham} after the proof of the proposition). 

These two proofs heavily rely on ideas used by Banyaga in \cite{Banyaga07}: In particular, the proof of the proposition above is a generalization of a part of Banyaga's proof of non-degeneracy (namely, his proof of \eqref{eq:B-thm}, see \S\ref{sec:equiv-bany-hofers-dist} below).

\begin{proof}[Proof of Proposition \ref{prop:non-degeneracy}]
If $c=0$ and $\mu$ is not injective, $d_{(\mu,c)}$ is degenerate since there exists a non-trivial path of symplectomorphisms with length 0. 

Now, assume that $c\neq 0$. We want to prove that if $d_{(\mu,c)}(\id,\psi)=0$, then $\psi\in \ham{(M,\omega)}$. In order to do that, we adapt an idea due to Banyaga in \cite{Banyaga07} and we prove (the stronger fact) that any path of symplectomorphisms from $\id$ to such a $\psi$ with small enough length has flux 0).

Recall from \eqref{eq:def-flux-morph} that the flux of a path $\Psi\in\m P_\id\symp{\momeg}$ is defined as
\begin{align*}
  \flux(\Psi) = \int_0^1 [\alpha(\Psi)_t]\, dt \quad\in \, H^1_\mathrm{dR}(M,\bb R)
\end{align*}
and only depends on the homotopy class (with fixed ends) of $\Psi$. Since $\mu$ takes its values in $B^1(M)$, the 1--forms $\alpha(\Psi)_t$ and $\alpha(\Psi)_t - \mu(\alpha(\Psi)_t)$ represent the same cohomology class (for every $t$). 

Let $\psi$ be a symplectomorphism such that $d_{(\mu,c)}(\id,\psi)=0$. For all $\eps>0$, there exists $\Psi^\eps$ a path of symplectomorphisms running from $\id$ to $\psi$ with length $\ell_{(\mu,c)}(\Psi) < c\cdot\eps$. In particular, 
\begin{align*}
  \left\| \int_0^1 \alpha(\Psi^\eps)_t - \mu(\alpha(\Psi^\eps)_t) \,dt \right\|_{L^2}  \leq \int_0^1 \| \alpha(\Psi^\eps)_t - \mu(\alpha(\Psi^\eps)_t) \|_{L^2} \,dt \leq \frac{1}{c} \, \ell_{(\mu,c)}(\Psi) < \eps
\end{align*}
that is, $\flux(\Psi^\eps)$ admits a representative, $\beta_{\eps}=\int_0^1 \alpha(\Psi^\eps)_t - \mu(\alpha(\Psi^\eps)_t) \,dt$, whose $L^2$--norm is bounded above by $\eps$. 

Pick any two $\eps$ and $\eps'$ and consider $\gamma_{\eps,\eps'}=\Psi^\eps\#\bar\Psi^{\eps'}$ (namely, a smooth reparametrization of the loop of symplectomorphisms obtained by first going from the identity to $\psi$ along $\Psi^\eps$ and coming back along $\Psi^{\eps'}$ with reversed orientation). For all $\eps$ and $\eps'$, $\flux(\gamma_{\eps,\eps'})$ lies in the flux group, $\Gamma_M=\flux(\pi_1(\sympo{\momeg}))$, and admits a representative, $\beta_{\eps}-\beta_{\eps'}$, with $L^2$--norm smaller than $\eps+\eps'$. 

Since $\Gamma_M$ is discrete by Ono's Theorem \cite{Ono06}, there are finitely many harmonic $1$--forms with $L^2$--norm smaller than $1$ (for example) and whose cohomology class lies in $\Gamma_M\backslash \{0\}$. Denoting by $\eps_0>0$ the minimum of the $L^2$--norm of these (finitely many) $1$--forms, we can deduce that as soon as $\eps+\eps'<\eps_0$, $\flux(\gamma_{\eps,\eps'})=0$ (because the $L^2$--norm of the harmonic representative of $\flux(\gamma_{\eps,\eps'})$ is bounded above by the $L^2$--norm of any other representative of the same class and is thus smaller than $\eps_0$).

Since the flux is a morphism, $\flux(\Psi^\eps)= \flux(\Psi^{\eps'})$ for $\eps+\eps'<\eps_0$. Hence, $\int_0^1 \alpha(\Psi^{\eps'})_t - \mu(\alpha(\Psi^{\eps'})_t) \,dt$ provide representatives of $\flux(\Psi^\eps)$ of arbitrarily small $L^2$--norm, that is, the $L^2$--norm of the harmonic representative of $\flux(\Psi^\eps)$ is arbitrarily small and thus has to be 0. This amounts to $\flux(\Psi^\eps)=0$ for any $\eps<\eps_0$. This in particular implies that $\psi\in\ham{(M,\omega)}$.
\end{proof}

Now we can proceed with the proof of Proposition \ref{coro:dist2Ham}, which is partly (the non-degeneracy of $\Delta_{(\mu,c)}$) a corollary of the proof above.
\begin{proof}[Proof of Proposition \ref{coro:dist2Ham}]
  Assume that $\Delta_{(\mu,c)}(\psi)=0$. By definition, this means that for every $\eps>0$, there exists a Hamiltonian diffeomorphism $\varphi_\eps$ such that $d_{(\mu,c)}(\varphi_\eps,\psi)<\eps$. Thus, we get for all $\eps'$ the existence of a path of symplectomorphisms $\Theta_{\eps,\eps'}$ connecting $\id$ to $\varphi_\eps$ with length $\ell_{(\mu,c)}(\Theta_{\eps,\eps'})<\eps+\eps'$.

This in particular implies that $\flux(\Theta_{\eps,\eps'})$ admits a representative, $\int \alpha(\Theta_{\eps,\eps'})-\mu(\alpha(\Theta_{\eps,\eps'}))$, with $L^2$--norm smaller than $\eps+\eps'$. Then, as in the proof above, by choosing $\eps$ and $\eps'$ small enough, we conclude that $\flux(\Theta_{\eps,\eps'})=0$ and thus that $\varphi_\eps^{-1}\psi\in \ham{\momeg}$ and thus so is $\psi$.

Concerning the bi-invariance, the left invariance of $\Delta_{(\mu,c)}$ comes from the left-invariance of $d_{(\mu,c)}$ (even over $\sympo{\momeg}$): $d_{(\mu,c)}(\phi\varphi,\phi\psi)=d_{(\mu,c)}(\varphi,\psi)$ for any three symplectomorphisms $\phi$, $\varphi$ and $\psi$. So, restricting to the case when $\phi$ and $\varphi$ are Hamiltonian diffeomorphisms, we have
\begin{align*}
  \Delta_{(\mu,c)} (\psi) &= \inf \{ d_{(\mu,c)}(\varphi,\psi)  \,|\;\varphi\in\ham{\momeg} \} \\ &= \inf \{ d_{(\mu,c)}(\phi\varphi,\phi\psi)  \,|\; \phi\varphi\in\ham{\momeg} \} =   \Delta_{(\mu,c)} (\phi\psi)
\end{align*}
for any symplectomorphism $\psi$ and Hamiltonian diffeomorphism $\phi$. 

Now, since $\ham{\momeg}$ is a normal subgroup of $\sympo{\momeg}$ (and with the same notation) $\phi'=\psi\phi\psi^{-1}$ is Hamiltonian and we conclude, by left invariance, that:
\begin{align*}
  \Delta_{(\mu,c)} (\psi\phi) =  \Delta_{(\mu,c)} (\phi'\psi) =  \Delta_{(\mu,c)} (\psi) \;.
\end{align*}
Thus $\Delta$ is bi-invariant and we can define for any $[\psi] \in \sympo{\momeg}/\ham{\momeg}$, $\delta_{(\mu,c)}([\psi])$ as $\Delta_{(\mu,c)}(\psi')$ for any representative $\psi'$ of $[\psi]$.

The $\ham{\momeg}$--invariance above immediately implies that $\Delta(\psi\psi')=\Delta(\psi'\psi)$ for any two $\psi$ and $\psi'\in\sympo{\momeg}$ since $\ham{\momeg}=[\sympo{\momeg},\sympo{\momeg}]$, the commutator subgroup of $\sympo{\momeg}$.

Finally, for the triangle-type inequality, fix $\psi$, $\psi'\in\sympo{\momeg}$ and $\phi\in\ham{\momeg}$. Notice that for any $\phi'\in\ham{\momeg}$, by triangle inequality on $d_{(\mu,c)}$, we have
\begin{align*}
  d_{(\mu,c)}(\phi,\psi\psi') \leq d_{(\mu,c)}(\phi,\psi\phi') + d_{(\mu,c)}(\psi\phi',\psi\psi') \; .
\end{align*}
By left-invariance $d_{(\mu,c)}(\psi\phi',\psi\psi')=d_{(\mu,c)}(\phi',\psi')$. Since this term does not depend on $\phi$, by taking the infimum over all Hamiltonian diffeomorphisms, we get 
\begin{align*}
\underbrace{\inf_\phi d_{(\mu,c)}(\phi,\psi\psi')} &\leq   \underbrace{\inf_\phi d_{(\mu,c)}(\phi,\psi\phi')} +\; d_{(\mu,c)}(\phi',\psi')\\
\mbox{that is, }\qquad   \Delta_{(\mu,c)}(\psi\psi') \quad &\leq \quad\Delta_{(\mu,c)}(\psi\phi') \quad  + d_{(\mu,c)}(\phi',\psi')
\end{align*}
for all $\phi'\in\ham{\momeg}$. Now, by right-invariance of $\Delta$ with respect to composition with Hamiltonian diffeomorphisms, $\Delta_{(\mu,c)}(\psi\phi')=\Delta_{(\mu,c)}(\psi)$ and thus does not depend on $\phi'\in\ham{\momeg}$. Thus we deduce the desired inequality
\begin{align*}
\Delta_{(\mu,c)}(\psi\psi') &\leq \Delta_{(\mu,c)}(\psi) + \Delta_{(\mu,c)}(\psi')
 \end{align*}
by taking the infimum over all $\phi'$. 

Obviously $\delta_{(\mu,c)}$ also satisfies these relations on equivalence classes, since $\psi\psi'$ is a representative of $[\psi]\cdot[\psi']$ in $\sympo{\momeg}/\ham{\momeg}$.
\end{proof}

\section{Applications}\label{sec:applications}

In this section, we choose data $(\mu,c)$ with $c\neq 0$ and omit it from the notation. (The dependence on the specific choice of data is not essential.) All the results below can also be stated with $\Delta$ replaced by the quotient version $\delta$.

\subsection{Hamiltonian square of symplectomorphisms}

In view of its properties, it is not surprising that $\Delta$ provides obstructions for symplectomorphisms to be Hamiltonian. Here is a straightforward consequence of Proposition \ref{coro:dist2Ham}. 
\begin{coro}\label{coro:hamilt-square-sympl}
  Let $\psi\in\sympo{\momeg}$. If $\psi^2$ is Hamiltonian, $\Delta(\psi)=\Delta(\psi^{-1})$. 
\end{coro}

\begin{proof}
  For $\psi^2\in\ham{\momeg}$, $\Delta(\psi) = \Delta(\psi^2\psi^{-1}) = \Delta(\psi^{-1})$ by invariance of $\Delta$.
\end{proof}

\subsection{Natural mappings of $\sympo{\momeg}$}

The \emph{Bounded isometry conjecture} which we mentioned in the introduction, follows from the viewpoint adopted by Lalonde and Polterovich in \cite{LalondePolterovich97}, where symplectomorphisms are seen as isometries of the Hamiltonian diffeomorphism group with respect to Hofer's distance. 

Indeed, as mentioned above, $\ham{\momeg}$ is a normal subgroup of $\symp{\momeg}$, that is, $\psi\varphi\psi^{-1}$ is Hamiltonian as soon as $\varphi$ is and Hofer's distance is invariant on conjugacy classes, in the sense that: $d_\mathrm{Hofer}(\id,\varphi)=d_\mathrm{Hofer}(\id,\psi\varphi\psi^{-1})$.

So they defined the map
\begin{align*}
  C\co \symp{\momeg} \lra \mathrm{Isom}(\ham{\momeg},d_\mathrm{Hofer}) \;,\qquad \psi \longmapsto C_\psi
\end{align*}
with $C_\psi(\varphi)=\psi\varphi\psi^{-1}$. They declared such a $C_\psi$ to be \emph{bounded} if $d_\mathrm{Hofer}(\varphi,\psi\varphi\psi^{-1})$ were bounded from above (independently of $\varphi\in\ham{\momeg}$). They proved that $C_\psi$ is bounded (if and) only if $\psi$ is Hamiltonian for certain symplectic manifolds and they conjectured that this holds for general compact symplectic manifolds. 

Of course, $C_\psi$ can easily be extended by the same formula to a group endomorphism of the whole symplectomorphism group and we also denote by $C$ its restriction to the connected component of the identity:
\begin{align*}
  C\co \sympo{\momeg} \lra \mathrm{End}(\sympo{\momeg},\circ) \;,\qquad \psi \longmapsto C_\psi \;.
\end{align*}
We denote by $S_\Delta(r)$ the $\Delta$--sphere of radius $r$:
\begin{align*}
  S_\Delta(r) = \{ \psi\in\sympo{\momeg}\,|\; \Delta(\psi)=r  \} \;.
\end{align*}
Then we have the following obvious corollary of Proposition \ref{coro:dist2Ham}, which tells that the endomorphisms $C_\psi$ are ``$\Delta$--isometries''.
\begin{coro}\label{coro:conjug-pres-shperes}
For any $\psi\in\sympo{\momeg}$ and any radius $r$, $C_\psi(S_\Delta(r)) = S_\Delta(r)$.
\end{coro}

\begin{proof}
$\Delta(\psi\psi')=\Delta(\psi'\psi)$ for any two elements of $\sympo{\momeg}$.
\end{proof}

Note that $S_\Delta(0)=\ham{\momeg}$ and in that case, $C_\psi(S_\Delta(0)) = S_\Delta(0)$ only reflects the fact that $\ham{\momeg}$ is normal in $\sympo{\momeg}$.

Finally, let us note that the (left or right) multiplication is ``$\Delta$--bounded'', that is
\begin{prop}\label{prop:bounded-mult}
  Let $\psi\in\sympo{\momeg}$. For all $\phi\in\sympo{\momeg}$:
$$|\Delta(\psi\phi)-\Delta(\phi)|\leq R(\psi) \qquad \mbox{ with }\;  R(\psi)=\max\{\Delta(\psi),\Delta(\psi^{-1})\} \; .$$
\end{prop}

\begin{proof}
  Assume $\Delta(\psi\phi)\geq\Delta(\phi)$, then
\begin{align*}
  |\Delta(\psi\phi)-\Delta(\phi)| = \Delta(\psi\phi)-\Delta(\phi) \leq \Delta(\psi) + \Delta(\phi)-\Delta(\phi) = \Delta(\psi) \;.
\end{align*}
Otherwise, $\Delta(\psi\phi)\leq\Delta(\phi)$ and then
\begin{align*}
  |\Delta(\psi\phi)-\Delta(\phi)| = \Delta(\phi)-\Delta(\psi\phi) &=  \Delta(\psi^{-1}(\psi\phi))-\Delta(\psi\phi) \\ 
  &\leq \Delta(\psi^{-1})+\Delta(\psi\phi)-\Delta(\psi\phi) =  \Delta(\psi^{-1}) \;.
\end{align*}
Either way, $|\Delta(\psi\phi)-\Delta(\phi)|$ is bounded above by $R(\psi)$.
\end{proof}

\subsection{Application to flux groups}\label{sec:appl-flux-groups}

First, we recall a few obvious observations and explain the facts mentioned in the introduction which motivate Theorem \ref{theo:appli-flux-groups}. Then we prove the theorem. 

Let $(M,\omega)$ and $(M',\omega')$ be compact symplectic manifolds. The Cartesian product $(M\times M',\omega\oplus \omega')$ is symplectic and there is an obvious group morphism
\begin{align*}
  \varpi\co \sympo{\momeg} \times \sympo{\momegp} \ra \sympo{(M\times M',\omega\oplus \omega')} \,, \quad (\phi,\psi) \mapsto \phi\times \psi
\end{align*}
which induces a map on fundamental groups
\begin{align*}
  \xymatrix@C=.3cm{\relax
 \pi_1(\sympo{\momeg}) \times \pi_1(\sympo{\momegp}) \ar@{}[r]|{\hspace{.3cm}\simeq}  & \pi_1( \sympo{\momeg} \times \sympo{\momegp})\ar[d] \\
& \pi_1(\sympo{(M\times M',\omega\oplus \omega')})
}
\end{align*}

The tangent space of the product splits, i.e $T_{(x,x')}(M\times M')\simeq T_xM\oplus T_{x'}M'$, and via this identification, for $\Xi=(\Phi,\Psi)$ a path of split symplectomorphisms of the product starting at $\id$, $X^\Xi_t(x,x') = X^\Phi_t(x)\oplus X^\Psi_t(x')$ and thanks to the choice of symplectic structure on the product, we get $ \alpha(\Xi)_t = \alpha(\Phi)_t \oplus \alpha(\Psi)_t$ in $Z^1(M)\oplus Z^1(M')\subset Z^1(M\times M')$. When $\Xi$ is a loop, that is, $\xi_1=\xi_0=\id$, this immediately leads to
$$\flux([\Xi]) = \flux([\Phi])\oplus \flux([\Psi])$$
which in turn gives
$$ \Gamma_M\times \Gamma_{M'} \simeq \Gamma_M\oplus \Gamma_{M'} \subset \Gamma_{M\times M'}$$
under the identification $ H^1_\mathrm{dR}(M\times M' ) \simeq H^1_\mathrm{dR}(M) \oplus  H^1_\mathrm{dR}(M')$ given by K\"unneth's formula.

Next recall that the flux group and the flux morphism are involved in a short exact sequence
\begin{align}\label{eq:ses-flux}
    \xymatrix{\relax
 0 \ar[r] &\ham{\momeg} \ar[r] &\sympo{\momeg} \ar[r]^\flux &H^1_\mathrm{dR}(M)/\Gamma_M \ar[r] &0
}
\end{align}
where the first (non-trivial) morphism is the inclusion and the second one, $\flux$, is induced by the flux morphism.

Now, it is easy to prove the following lemma (which appears in \cite{CamposPedroza10}).
\begin{lemm}
  With the same notation, $\Gamma_M\times \Gamma_{M'} \simeq \Gamma_{M\times M'}$ if and only if there are no non-Hamiltonian diffeomorphisms $\phi$ and $\psi$ (of resp. $\momeg$ and $\momegp$) whose product $\phi \times \psi$ is a Hamiltonian diffeomorphism of $(M\times M',\omega\oplus\omega')$.
\end{lemm}

\begin{rema}
  This easy lemma is quite remarkable since it tells us that the ``difference'' between the flux group of the product and the product of the flux groups comes from certain \textit{split} Hamiltonian diffeomorphisms of the product. (In view of the definition of the flux group, one might have thought that it would actually come from non-split symplectomorphisms.)
\end{rema}

\begin{proof}
  Derive from \eqref{eq:ses-flux} the following commutative diagram
\begin{align*}
  \xymatrix{\relax
     \sympo{\momeg}\times \sympo{\momegp} \ar[d]_\varpi \ar[r] &H^1_\mathrm{dR}(M)/\Gamma_M\times H^1_\mathrm{dR}(M')/\Gamma_{M'} \ar[d]^p\ar[r] &0\\
     \sympo{(M\times M',\omega\oplus\omega')} \ar[r] &H^1_\mathrm{dR}(M\times M')/\Gamma_{M\times M'} \ar[r] &0
  }
\end{align*}
where $p$ is the composition
\begin{align*}
\xymatrix@C=.8cm{\relax
  H^1_\mathrm{dR}(M)/\Gamma_M\times H^1_\mathrm{dR}(M')/\Gamma_{M'} \ar[rd]_p \ar@{}[r]|{\hspace{0.3cm}\simeq} & H^1_\mathrm{dR}(M\times M')/\Gamma_{M}\oplus \Gamma_{M'}\ar[d]^{p'} \\
  & H^1_\mathrm{dR}(M\times M')/\Gamma_{M\times M'}
}
\end{align*}
with $p'$ the obvious projection: $\Gamma_{M}\oplus \Gamma_{M'}$ being a subgroup of $\Gamma_{M\times M'}$. Now these groups coincide if and only if $p'$ is injective, that is, if and only if $p$ is injective.

Since the flux is surjective and has $\mathrm{Ham}$ as kernel (as indicated by \eqref{eq:ses-flux}), the isomorphism $\Gamma_{M}\oplus \Gamma_{M'}\simeq\Gamma_{M\times M'}$ holds if and only if there is no split Hamiltonian diffeomorphism of the product, $\phi\times\psi$, with $\phi$ and $\psi$ non-Hamiltonian. 
\end{proof}

Let us now prove Theorem \ref{theo:appli-flux-groups} which gives information on this obstruction, namely, that if such a split Hamiltonian diffeomorphism exists, then the distance from $\psi$ to the Hamiltonian diffeomorphism group is bounded from below by a constant which only depends on $\phi$.
\begin{proof}[Proof of Theorem \ref{theo:appli-flux-groups}]
  Let $\phi$ and $\psi$ be symplectomorphisms of $\momeg$ and $\momegp$ respectively. First notice that if $\phi\times\psi$ is Hamiltonian, then $\psi$ and $\phi$ have the same nature, that is, they are \emph{both} either Hamiltonian or non-Hamiltonian. (An easy way to see this is by additivity of the flux: $\flux(\phi\times\psi)=\flux(\phi)+\flux(\psi)$, and the fact that the kernel of the flux consists of the Hamiltonian diffeomorphisms.) 

Notice that this already tells that when $\phi$ is not Hamiltonian, the map
\begin{align*}
  \sympo{(M',\omega')} \longrightarrow \sympo{(M\times M',\omega\oplus\omega')} \;, \qquad \psi \longmapsto \phi\times \psi
\end{align*}
takes its values in the non-Hamiltonian diffeomorphisms of the product when restricted to $\ham{\momegp}$. Let us prove the existence of a $\Delta$--neighborhood of $\ham{\momegp}$ for which this still holds. 

We need to choose a pair $(\mu,c)$ for each manifold involved. Let fix $c_M=c_{M'}=c_{M\times M'}\neq 0$ (and then forget about these in the notation). We choose $\mu_M$ and $\mu_{M'}$, endomorphisms from $Z^1(M)$ to $B^1(M)$ and from $Z^1(M')$ to $B^1(M')$ respectively and define $\mu_{M\times M'}$ in such way that $\mu_{M\times M'}=\mu_M\oplus\mu_{M'}$ on $Z^1(M)\oplus Z^1(M')\subset Z^1(M\times M')$ (notice that $\im(\mu_{M\times M'})\subset B^1(M)\oplus B^1(M')\subset B^1(M\times M')$). We will simply denote $\mu_{M\times M'}$ by $\mu$.

We also choose metrics $g_M$ and $g_{M'}$ on $M$ and $M'$ and define a metric $g$ on $M\times M'$ accordingly. Notice that the direct sum $Z^1(M)\oplus Z^1(M')\subset Z^1(M\times M')$ is orthogonal with respect to the $L^2$--norm induced by this choice. Finally, we choose a norm on $B^1(M\times M')$ which induces norms on $B^1(M)$ and $B^1(M')$.

Assume that $\phi\times \psi$ is Hamiltonian but that $\phi$ (and thus $\psi$) is not. Then, by the triangle inequality property satisfied by $\Delta$, 
\begin{align*}
  0< \Delta_\mu (\phi^{-1}\times\id_{M'}) &= \Delta_\mu ((\phi^{-1}\times\psi^{-1})\circ (\id_M\times\psi))\\
&\leq  \Delta_\mu (\phi^{-1}\times\psi^{-1}) +  \Delta_\mu (\id_M\times\psi) = \Delta_\mu (\id_M\times\psi)
\end{align*}
since $\phi^{-1}\times\psi^{-1}=(\phi\times\psi)^{-1}$ is Hamiltonian. Moreover, we also have
\begin{align*}
  0< \Delta_\mu (\id_M\times\psi) &= \Delta_\mu ((\phi\times\psi)\circ (\phi^{-1}\times \id_{M'}))\\
&\leq  \Delta_\mu (\phi\times\psi) +  \Delta_\mu ((\phi^{-1}\times \id_{M'}) = \Delta_\mu (\phi^{-1}\times\id_{M'})
\end{align*}
and thus $\Delta_\mu (\id_M\times\psi)=\Delta_\mu (\phi^{-1}\times\id_{M'})$, which is non-zero since $\psi$ is not Hamiltonian\footnote{We suspect that this equality, from which the proof follows, leads to other facts of interest.}. We put $\eps(\phi)=\Delta_\mu (\phi^{-1}\times\id_{M'})$.

Let $\theta\in \sympo{\momegp}$. To any path $\Theta$ in $\sympo{\momegp}$ with $\theta_0=\id_{M'}$ and $\theta_1=\theta$, corresponds $\id_M\times \Theta$, path of symplectomorphisms of the product. Thanks to our choices of homomorphisms, $\mu(\alpha(\id_M\times \Theta))=\mu_{M'}(\alpha(\Theta))$ and in view of our choices (of metrics, norms and constants), corresponding paths, $\Theta$ and $\id_M\times \Theta$, have same length, namely:
\begin{align*}
  \ell_\mu(\id_M\times\Theta) = \ell_{\mu_{M'}}(\Theta) \;.
\end{align*}
Thus, for all $\varphi\in \ham{\momegp}$:
\begin{align*}
   d_\mu(\id_M\times \varphi, \id_M\times\psi) = d_\mu(\id_{M\times M'}, \id_M\times\varphi^{-1}\psi) \leq \ell_{\mu}(\id_M\times\Theta) = \ell_{\mu_{M'}}(\Theta) 
\end{align*}
for any $\Theta\subset\sympo{\momegp}$ connecting $\id_{M'}$ to $\varphi^{-1}\psi$. By taking the infimum over all such paths $\Theta$ we derive the inequality $d_\mu(\id_M\times \varphi, \id_M\times\psi) \leq d_{\mu_{M'}}(\varphi,\psi)$ which in turn leads to
\begin{align*}
  \Delta_\mu(\id_M\times\psi) \leq d_\mu(\id_M\times \varphi, \id_M\times\psi) \leq d_{\mu_{M'}}(\varphi, \psi)
\end{align*}
for all $\varphi\in \ham{\momegp}$. Again, by taking the infimum over all Hamiltonian $\varphi$'s, we conclude that if $\phi\times\psi$ is a Hamiltonian diffeomorphism of the product, necessarily
\begin{align*}
   \Delta_{\mu_{M'}}(\psi)\geq  \Delta_\mu(\id_M\times\psi)= \eps(\phi)  
\end{align*}
which concludes the proof of the theorem.
\end{proof}

\section{Remarks about Condition $(\dagger)$}\label{sec:ConditionDagger}


In this section we discuss the geometric meaning of condition $(\dagger)$, and deliver some evidence toward the fact that operators satisfying $(\dagger)$ do not exist. This might seem disappointing at first, since $(\dagger)$ implies several nice properties on the lengths and distance functions. However, we will prove at the end of this section (see \S \ref{sec:proof-prop-daggerHamAPoint}) the following fact.
\begin{prop} \label{prop:dagger-HamAPoint}
	Let $\mu$ satisfy $(\dagger)$ and $n_{(\mu,c)}$ be a $\ham{\momeg}$--invariant norm. Either the induced distance, $d_{(\mu,c)}$, vanishes on $\mathrm{Ham}(M,\omega)$ or $\mu$ is injective.
\end{prop}
Now recall (from Remark \ref{rema:injective-mu}) that injective $\mu$'s lead to ``unnatural'' distances (as extensions to $\sympo{\momeg}$ of distances on $\ham{\momeg}$ for manifolds $M$ with non-trivial first De Rham cohomology group), because such operators act non-trivially on an infinite dimensional subspace of $B^1(M)$. 

Then the proposition may be thought\footnote{This is not equivalent since a bi-invariant distance does not necessarily come from an invariant norm. Moreover, we only consider the action of $\ham{\momeg}$ on $B^1(M)$ instead of $\sympo{\momeg}$.} as: When $H_{\mathrm{dR}}^1(M)\neq 0$, the only pseudo-distance on $\ham{\momeg}$ which can be extended to a bi-$\ham{\momeg}$--invariant distance $d_{(\mu,c)}$ on $\sympo{\momeg}$ is the trivial one.

Let us mention that results concerning extensions of distances from $\ham{\momeg}$ to $\sympo{\momeg}$ also appear in Han \cite{Han09} and that a slightly different version of the proposition above can be deduced directly from recent work by Buhovsky and Ostrover \cite{BuhovskyOstrover10} (partially relying on previous observations contained in \cite{OstroverWagner05}).


\subsection{Additional properties under Condition $(\dagger)$}

As mentioned above in Remark~\ref{rema:ConditionDagger}, when $c=0$, the fact that the endomorphism $\mu$ satisfies 
          \begin{align*}
            \forall\alpha\in Z^1(M), \, \forall\phi\in\symp_0{\momeg}, \quad n^B(\mu(\alpha))=n^B(\mu(\phi^*\alpha))  
            \tag{$\dagger$}
          \end{align*}
(which is equivalent to requiring that $n^B(\mu(\alpha))\leq n^B(\mu(\phi^*\alpha))$ for any $\alpha$ and any $\phi$) has several interesting consequences. In particular, it leads to the equality $\ell_{(\mu,c)}(\alpha(\Phi^{-1} \circ \Psi)) = \ell_{(\mu,c)}(\alpha(\Phi) - \alpha(\Psi))$ which in turn leads to the symmetry of $d_{(\mu,c)}$, that is,
\begin{align*}
  d_{(\mu,c)}(\phi,\psi) =  d_{(\mu,c)}(\psi,\phi)
\end{align*}
for all symplectomorphisms $\phi$ and $\psi$ isotopic to the identity. Thus, under Condition $(\dagger)$, $d_{(\mu,c)}$ is a pseudo-distance. 

It also leads to the fact that, when $c=0$, the lengths of the different ``products'' of paths can be related as follows: $\ell_{(\mu,0)}(\Phi\circ\Psi) \leq \ell_{(\mu,0)}(\Phi\ast_l\Psi) =\ell_{(\mu,0)}(\Phi\ast_r\Psi)$, that is, time-wise composition always generates paths shorter than the ones generated by (left and right) concatenations.

Actually Condition $(\dagger)$ is stronger than the following one:
          \begin{align*}
              \forall\phi\in &\symp_0{\momeg},\; \exists\, c(\phi)\in\bb R \mbox{ s.t. }\\  &\forall\alpha\in Z^1(M),  \quad n^B(\mu(\phi^*\alpha))\leq c(\phi)n^B(\mu(\alpha))  
            \tag{weak--$\dagger$}
          \end{align*}
which obviously corresponds to the fact that the pullback action of $\sympo{\momeg}$ on $Z^1(M)$ is continuous with respect to $n^B\circ\mu$. In this perspective, Condition $(\dagger)$ simply reads that $\sympo{\momeg}$ acts isometrically.
\begin{rema}
  Notice that if $\mu$ is non-trivial and satisfies Condition (weak--$\dagger$), then for all $\phi \in \sympo{\momeg}$, $c(\phi)c(\phi^{-1})\geq 1$ and that (\textit{if we define $c(\phi)$ as the \emph{smallest} constant such that \emph{(weak--$\dagger$)} is satisfied then}) $c(\phi^n) \leq c(\phi)^n$. 

Notice also that, for a fixed $\mu$ -- which does not necessarily satisfy Condition (weak--$\dagger$) --, the symplectomorphisms $\phi$ satisfying 
          \begin{align*}
              \forall\alpha\in Z^1(M),  \quad n^B(\mu(\phi^*\alpha))\leq c(\phi)n^B(\mu(\alpha))  
          \end{align*}
form a monoid in $\sympo{\momeg}$.
\end{rema}

Another way to look at these conditions is the following. The one-to-one correspondence \eqref{eq:correspondences-paths} between $\m P_\id\symp{\momeg}$ and $\m P Z^1(M)$ is in general not compatible with the algebraic structures on these spaces. As mentioned above, $\m P_\id\symp{\momeg}$ endowed with time-wise composition is a group and so is $\m P Z^1(M)$ endowed with time-wise addition. However, composition of paths coincides with a ``twisted'' non-commutative addition on $\m PZ^1(M)$, $+_{\!\!*}$, which we could abstractly define as
\begin{align*}
  \alpha +_{\!\!*} \beta = \alpha + \Psi^*\beta \qquad \mbox{where } \alpha(\Psi)=\alpha \;.
\end{align*}

Under Condition $(\mbox{weak--}\dagger)$, this correspondence
\begin{align*}
       \m C\co (\m P_\id\symp{\momeg},\circ) \lra (\m P Z^1(M),+_{\!\!*})
\end{align*}
is continuous with respect to the ``natural'' norms induced by our lengths above, in the sense that
\begin{align*}
 \ell_{(\mu,0)} (\Phi^{-1}\circ \Psi)  \leq c(\Phi) \, \ell_{(\mu,0)} (\alpha(\Phi)-\alpha(\Psi))
\end{align*}
and is an isometry under Condition $(\dagger)$, since we get
\begin{align*}
  \ell_{(\mu,0)} (\Phi^{-1}\circ \Psi)  = \ell_{(\mu,0)} (\alpha(\Phi)-\alpha(\Psi)) \;.
\end{align*}

Finally, Condition $(\mbox{weak--}\dagger)$ (resp. $(\dagger)$) amounts to the fact that ``taking the inverse'' is continuous (resp. is an isometry) on $\m P_{\id}\symp{\momeg}$, since $\alpha(\Phi^{-1})=-\Phi^*\alpha(\Phi)$.

Condition (resp. weak-) $(\dagger)$ does not only appear from the viewpoint of (continuity) isometry of usual operations on paths with respect to our length but also in terms of invariance, as indicated by Proposition \ref{prop:dagger-HamAPoint}. In particular, one can easily prove the fact that \textit{under condition $(\dagger)$, the restriction of $d_{(\mu,0)}$ to $\ham{\momeg}$ is bi-$\sympo{\momeg}$--invariant}.

\subsection{Toward the non-existence of operators satisfying $(\dagger)$}

For the reasons above, operators $\mu$ satisfying Condition $(\dagger)$ seem more natural for the construction we are describing here. However, we strongly believe that such operators do not exist and we will now justify the following precise conjecture.

\begin{conj}\label{conj:nonexist-dagger-mu}
  There is no non-injective, non-trivial endomorphism $\mu$ satisfying Condition $(\mbox{weak--}\dagger)$.
\end{conj}

By non-trivial, we rule out the operator $\mu=0$ which obviously satisfies $(\dagger)$ and, in case $Z^1(M) = B^1(M)$, $\mu = \id$ (which satisfies $(\dagger)$ when we choose the Hofer norm as norm $n^B$, for it is pullback invariant). We also suspect that the conjectured result holds if we drop the non-injectivity assumption, however, we have less supporting evidence in that case.

Our conjecture is motivated by the following observation:
\begin{lemm}\label{lemm:infinite-codim-Inv-subspaces}
  If $\mu$ satisfies $(\mbox{weak--}\dagger)$, its kernel is a subspace of $Z^1(M)$ which is \emph{invariant with respect to the pullback-action of $\sympo{\momeg}$} and which has \emph{infinite codimension}, as soon as $\mu\neq 0$.
\end{lemm}

\begin{proof}
  If $\mu$ satisfies $(\mbox{weak--}\dagger)$, the invariance property of $\ker \mu$ is clear. Now assume $\mu$ non-trivial, and fix a closed 1--form $\beta$ such that $\mu(\beta) \neq 0$. Since $\mu$ satisfies $(\mbox{weak--}\dagger)$, then necessarily $\phi^* \beta \notin \ker(\mu)$ for all $\phi \in \sympo{\momeg}$.

Choose a small open disc $U\subset M$. The group of symplectomorphisms compactly supported in $U$ is an infinite dimensional subgroup $\mathrm{Symp}_0^U(M,\omega) \subset \sympo{\momeg}$ which, in particular, contains all the results of flows of Hamiltonian vector fields. Since we can flow any point (of $U$) to another, it is clear that as soon as $\beta$ is not constant on $U$, it is possible to produce infinitely many linearly independent functions.

This provides a sequence $\{\phi_j\} \subset \mathrm{Symp}_0^U\momeg$ such that all $\phi_j^*\beta$ are mutually linearly independent which prevent $\ker\mu$ from having finite codimension.
\end{proof}

Now we explain why we suspect that \textit{any non-trivial subspace of $Z^1(M)$ which is invariant by the pullback action of $\sympo{\momeg}$ has finite codimension}. 

Notice that the exact 1--forms constitute an invariant subspace of $Z^1(M)$ and that, more generally, the subspace of $Z^1(M)$ generated by all the representatives of a given de Rham cohomology class (not necessarily 0) is invariant as well. We are convinced that these are essentially all the invariant subspaces or, equivalently, that any non-trivial, invariant subspace of $Z^1(M)$ contains $B^1(M)$ (and thus has finite codimension, smaller than $b^1(M)=\dim H_{\mathrm{dR}}^1(M)$).   

  The following lemma states that the intersection of any two non-trivial pullback-invariant subspaces of $Z^1(M)$ is non-empty and (thus) infinite dimensional. 

Let us denote by $\invg_\alpha=\langle G^*\alpha \rangle_{\bb R}$ the smallest $G$--pullback invariant subspace of $Z^1(M)$ which contains $\alpha$.
\begin{lemm}
  Let $\alpha$ and $\beta$ be any non-zero closed $1$--forms. There exists a non-trivial exact $1$--form $\gamma$ such that $\invd_\gamma \subset \invd_\alpha \cap \invd_\beta$. 
\end{lemm}

As a consequence, we get that for any closed $1$--form $\alpha$, $\invd_\alpha \cap B^1(M) \neq \{ 0 \}$. (However, this result can be derived from the facts that $B^1(M)$ has finite codimension and that the dimension of any pullback invariant subspace is infinite.)

\begin{proof}
   Let $\alpha$ and $\beta$ be non-zero closed $1$--forms. Since $\alpha$ and $\beta$ are regular on some open sets of $M$, up to choosing smaller open sets and taking the pullback of $\beta$ by some appropriate diffeomorphism of $M$, we can assume that both forms are regular on some small open chart $U$ of $M$.

Then, there exist diffeomorphisms $\phi_\alpha$ and $\phi_\beta$ from the unit disc $\bb D^n \subset \bb R^n$ to $V$ (an open set such that $\overline{V}\subset U$) and satisfying $\phi_\alpha^*\alpha=\phi_\beta^*\beta=dx_1$ on $\bb D^n$ (the differential of the projection on the first coordinate).

   Let us consider the diffeomorphism of $\bb D^n$, $\phi = \phi_\alpha^{-1}\circ \phi_\beta$. First we notice that $\phi_\alpha$ and $\phi_\beta$ can be chosen such that $\phi$ preserves the orientation. (Otherwise, replace $\phi_\beta$ by $\phi_\beta' = \phi_\beta\circ (x_2\mapsto -x_2)$, notice that $\phi_\beta^*\beta = (\phi_\beta')^*\beta$.) Since the diffeomorphism group of the disc has two connected components, $\phi$ is isotopic to the identity. This immediately amounts to the fact that the diffeomorphism of $V$, $\psi = \phi_\beta \circ \phi_\alpha^{-1}$, is also isotopic to the identity.

   Thus we can extend $\psi$ to $U$ first, in such a way that $\psi=\id$ near $\del U$ and then extend it to a diffeomorphism of $M$ which is the identity on $M\backslash U$. The resulting diffeomorphism, also denoted $\psi$, is isotopic to the identity.

   Now pick $\varphi$, any non-trivial diffeomorphism of $M$, isotopic to the identity, with support included in $V$. Notice that $\varphi^*\alpha - \alpha = \varphi^*(\psi^*\beta)-\psi^*\beta$. (Indeed, $\alpha = \psi^*\beta$ on $V$ and $\varphi^*\alpha - \alpha = \varphi^*(\psi^*\beta)-\psi^*\beta=0$ on $M\backslash V$.) This $1$--form is thus in the intersection $\invg_\alpha \cap \invg_\beta$.
\end{proof}

Obviously the lemma above is weaker than what we intend to prove and conjectured. (Moreover, it is proved for $G=\diffo$ and not $\sympo{\momeg}$.) Notice, however, that it is sufficient in order to deduce the following corollary.
\begin{coro}\label{theo:non-trivial-intersection}
  Let $M$ be a smooth manifold. The pullback-action of $\diffo$ on $Z^1(M)$ is irreducible (even though $Z^1(M)$ is not simple). 
\end{coro}

\subsection{Proof of Proposition \ref{prop:dagger-HamAPoint}}\label{sec:proof-prop-daggerHamAPoint}

Assume that $\mu$ is an operator which satisfies 
          \begin{align*}
            \forall\alpha\in Z^1(M), \, \forall\phi\in\ham{\momeg}, \quad n^B(\mu(\alpha))=n^B(\mu(\phi^*\alpha))  
            \tag{$\dagger'$}
          \end{align*}
(that is, we only require $\mu$ to satisfy Condition $(\dagger)$ with respect to the pullback action on $1$--forms by \textit{Hamiltonian} symplectomorphisms), then its kernel is a $\ham{\momeg}$--invariant subspace of $Z^1(M)$. Then, as in the proof of Lemma \ref{lemm:infinite-codim-Inv-subspaces}, if $\mu$ is not injective, its kernel has infinite dimension and thus intersects $B^1(M)$ non-trivially. Then, the distance $d_{(\mu,c)}$ on $\sympo{\momeg}$ restricts to a bi-$\ham{\momeg}$--invariant distance on $\ham{\momeg}$ which has to be degenerate. 

Thus, the restricted distance has to identically vanish for the null-set of a bi-invariant distance $d$, that is,
\begin{align*}
 \mathrm{Null}(d) =  \{ \phi\in\ham{\momeg} |\, d(\id,\phi)=0 \}  
\end{align*}
is a normal subgroup of $\ham{\momeg}$ and $\ham{\momeg}$ is known to be simple since Banyaga's celebrated paper \cite{Banyaga78}.

\section{Examples}\label{sec:examples}

\subsection{$\diffo$ as endomorphisms from $Z^1(M)$ to $B^1(M)$}

It is quite remarkable that there exists an injective mapping from the group of diffeomorphisms isotopic to the identity of $M$ to the vector space of morphisms from $Z^1(M)$ to $B^1(M)$ given as follows
\begin{align*}
  \diffo \lra \mathrm{Hom}(Z^1(M),B^1(M)) \;, \qquad f\longmapsto [\mu_f \co \alpha \mapsto \alpha-f^*\alpha]  \;.
\end{align*}

\begin{rema}
The mapping $f\mapsto \mu_f$ obviously satisfies the following properties.  
\begin{enumerate}
\item $\mu_f=\mu_g$ if and only if $f=g$.
\item In general, $\mu_f$ is not injective. 
\item $\mu_{f\circ g}=\mu_f+\mu_g(f^*) = \mu_g + g^*(\mu_f)$ and $\mu_f-\mu_g = g^*(\mu_f)-\mu_g(f^*)$.
\item $\mu_{f_1\circ g} - \mu_{f_2\circ g} = g^*(\mu_{f_1} - \mu_{f_2})$ and $\mu_{f\circ g_1} - \mu_{f\circ g_2} = (\mu_{g_1} - \mu_{g_2})(f^*)$.
\item $(g^{-1})^*\mu_g = \mu_g((g^{-1})^*) = - \mu_{g^{-1}}$.
\end{enumerate}
\end{rema}
Notice that $\mu_f$ can be interpreted as a ``localization'' in the sense that the $1$--forms which are obtained via $\mu_f$ vanish outside of the support of $f$ (where $f$ is the identity). However, inside the support of $f$, the resulting 1--form is in general extremely different from $\alpha$.

\subsection{$C^\infty(M)$ as endomorphisms from $Z^1(M)$ to $B^1(M)$}

Not only $\diffo$ provides linear operators from $Z^1(M)$ to $B^1(M)$, but so does $C^\infty(M)$. Indeed, given an \textit{autonomous} Hamiltonian $H \in C^\infty(M)$, we can define an operator $\mu_H$ by putting 
$$ \mu_H\co Z^1(M)\lra B^1(M) \;, \qquad  \alpha \longmapsto d(\alpha(X_H))\;.$$
(Observe that for a non-autonomous $H$, this process would induce a path of $\mu$'s and we would loose the compatibility with reparametrizations.)

\subsection{Banyaga's construction}\label{sec:example-banyaga}

When we specialize our construction to particular data, we are able to precise some aspects of our study. We do so in the case studied by Banyaga since it is the case which motivated this work. 

In what follows, we first sketch Banyaga's construction and establish that we indeed generalize it. Then we make a remark concerning the length of concatenations of paths, in relation with \S\ref{subsection:timewise-composition-and-concatenations} above. Finally, we answer positively a question raised by Banyaga in \cite{Banyaga07}: The distance on the Hamiltonian diffeomorphism group obtained by restricting Banyaga's distance is indeed equivalent to Hofer's distance.

\subsubsection{The Hodge linear operator}

In \cite{Banyaga07}, Banyaga uses Hodge decomposition to define a norm on $Z^1(M)$ as follows. A metric $g$ on $M$ being fixed, recall that $Z^1\simeq \mathcal H \oplus \im d$, where $d$ is the de Rham differential and $\m H$ is the space of harmonic 1--forms with respect to $g$. (Indeed, the space $\im \delta$ appearing as third summand in Hodge decomposition of $p$--forms is automatically $0$ when restricted to \emph{closed} forms.)

Thus any closed 1--form $\alpha$ uniquely decomposes as $du_\alpha + h_\alpha$, with $h_\alpha \in\m H$. Banyaga chooses a norm $\|-\|$ on $\mathcal{H}\simeq H^1_{\mathrm{dR}}(M)$ and then defines the length of $\Phi \in \m P_\id\symp{(M,\omega)}$ as
$$ \ell^B(\Phi) = \int_0^1 \left( \osc u_{\alpha(\Phi)_t} + \| h_{\alpha(\Phi)_t} \|\right)dt$$
where $\osc f=\max f-\min f$ is the ``Hofer norm'' of the exact form $df$.

\begin{lemm}\label{lemm:generalizes-Banyaga}
    Our construction generalizes Banyaga's (in the sense that Banyaga's construction is (equivalent to) a particular case of ours).
\end{lemm}

\begin{proof}
With the notation above, define $\mu^g$ as $\mu^g(\alpha)=du_\alpha$ and choose the Hofer norm for $n^B$: The (pseudo-)distance induced by $\ell^B$ is equivalent to $n_{(\mu^g,1)}$ since $H^1_{\mathrm{dR}}(M)$ is finite dimensional.
\end{proof}

The length $\ell^B$ is a quite particular case of $\ell_{(\mu,c)}$'s  (even if {we restrict ourselves to $c=1$ and $n^B$ is the Hofer norm}) as indicated by the following lemma.
\begin{lemm}\label{lemm:closed 1-forms not harmonic}
Let $M$ be a Riemannian manifold such that $\dim H^1_{\mathrm{dR}} (M)\geq 1$. There exist closed $1$--forms which are not harmonic for any choice of metric on $M$.
\end{lemm}

This result is probably well-known, but we did not find it as such in the literature. It is also quite easy to prove and we include a (constructive) proof here for the reader's convenience. 
\begin{proof}
Let $M$ be a Riemannian manifold such that $\dim H^1_{\mathrm{dR}} (M)\geq 1$. Let $\alpha$ be a closed $1$--form which is not exact. First, choose a chart $(U,\psi)$ of $M$, that is an open subset $U$ and a diffeomorphism $\psi\co U\ra \psi(U)=B(0,\eps)\subset\bb R^n$. The $1$--form $(\psi^{-1})^*\alpha$ is closed in $B(0,\eps)$ and thus exact: $(\psi^{-1})^*\alpha = df$ for some $f\co B(0,\eps)\ra \bb R$. We extend $f\circ \psi$ in any way on $M$ and call it $f$. The $1$--form $\alpha'=\alpha-df$ is still closed but not exact and vanishes on $U$. 

Now, choose an open set of $M$, $U'$, whose closure is included in $U$, and a function $f'\co M\ra \bb R$ with support included in $U'$. We define a third 1--form $\alpha''=\alpha' + df'$ which coincides with $\alpha'$ everywhere but on $U'$ (where it coincides with $df'$).

This latter $1$--form cannot be harmonic for any metric $g$ on $M$. Indeed, for closed forms, being harmonic is equivalent to being in the kernel of the co-differential $\delta$. Assume that for some metric $g$, $\delta \alpha''=0$. Then for any function ($0$--form) $h$ on $M$, 
\begin{align*}
0= \langle h,\delta\alpha'' \rangle_g = \langle dh,\alpha'' \rangle_g = \int_M \alpha''(\nabla^g h) \;.
\end{align*}
Now for $h=f'$, we get 
\begin{align*}
\int_M \alpha''(\nabla^g f') &= \int_{U'} \alpha''(\nabla^g f') + \int_{M\backslash U'} \alpha''(\nabla^g f') \\
& = \int_{U'} \big( \alpha'(\nabla^g f') + df'(\nabla^g f') \big) = \int_{U'} g(\nabla^g f',\nabla^g f') > 0
\end{align*}
since $\nabla^g f'=0$ on $M\backslash U'$ and $\alpha'=0$ on $U'$. Thus, $\alpha''$ is not harmonic.
\end{proof}

\subsubsection{Explicit non length-minimizing paths}\label{sec:difference-concatenations-Banyagas-case}

In view of remarks on concatenations of \S\ref{subsection:timewise-composition-and-concatenations}, we are able to produce non $\ell^B$--minimizing paths of symplectomorphisms.

Assume that $H^1_{\mathrm{dR}}(M)\neq 0$ and choose any non-zero de Rham cohomology class of degree $1$, $[\alpha]$. Recall that there is a unique harmonic $1$--form $h_{[\alpha]}\in[\alpha]$. 

Consider the path of symplectomorphisms induced by $h_{[\alpha]}$, that is, defined as
$$\psi_0 = \id \qquad\mbox{and}\qquad \del_t\psi_t = X(\psi_t)  $$
where $X$ is defined as the unique vector field such that $\omega(X,\cdot\,)=h_{[\alpha]}$. We denote $\psi_1$ by $\psi$.

\begin{lemm}\label{lemm:concatenation-not-minimiz}
  Let $\varphi\in \sympo{\momeg}$, with $\varphi\neq\psi$. There is no $\ell^B$--minimizing path from $\id$ to $\varphi$ which ``factorizes'' (in terms of right concatenation) through $\Psi$. In other words, paths of the form $\Phi\ast_r \Psi$, with $\phi_1\neq\id$ cannot be $\ell^B$--minimal.
\end{lemm}

We denote by $d^B_{(\mu^g,1)}$ the distance on $\ham{\momeg}$ obtained by restriction of Banyaga's distance on $\sympo{\momeg}$. (This is motivated by Proposition \ref{prop:comparison-with-nb} and Lemma \ref{lemm:generalizes-Banyaga}.)

\begin{proof}
First notice that since any $\phi\in\symp_0{\momeg}$ is homotopic to the identity, $[\phi^*h_{[\alpha]}]=[\alpha]$. By unicity of $h_{[\alpha]}$, $\phi^*h_{[\alpha]}$  Hodge decomposes as $h_{[\alpha]} + du_\phi$ and then 
$$\ell^B(\phi^*h_{[\alpha]}) - \ell^B(h_{[\alpha]}) = \big( \| h_{[\alpha]} \| + \osc(u_{\phi})  \big) - \| h_{[\alpha]} \| = \osc(u_{\phi}) \; .$$
If $\phi\neq \id$, we deduce that $\ell^B(\phi^*h_{[\alpha]}) > \ell^B(h_{[\alpha]})$ which, together with \eqref{eq:equality-length-concatenation-right}, leads to
\begin{align*}
\ell^B(\Phi\ast_r \Psi)=\ell^B(\Phi) + \ell^B(\phi\Psi) > \ell^B(\Phi) + \ell^B(\Psi)
\end{align*}
for any $\Phi\in\m P_\id\symp$ connecting $\id$ to $\phi$.

Now, take any $\varphi\neq\psi$ and put $\phi=\varphi\psi^{-1}$. Notice that $\phi\neq\id$ and take $\Phi$ as above (any element of $\m P_\id\symp$ connecting $\id$ to $\phi$). We get
\begin{align*}
d^B_{(\mu^g,1)}(\id,\phi) + d^B_{(\mu^g,1)}(\id,\psi) \leq  \ell^B(\Phi) + \ell^B(\Psi) <  \ell^B(\Phi\ast_r \Psi) \; .
\end{align*}
Now, if $\Phi\ast_r\Psi$ were indeed minimizing, we would also have
\begin{align*}
  \ell^B(\Phi\ast_r \Psi) = d^B_{(\mu^g,1)}(\id,\phi\psi) \leq d^B_{(\mu^g,1)}(\id,\phi) + d^B_{(\mu^g,1)}(\id,\psi)
\end{align*}
by triangle inequality. We get a contradiction and $\Phi\ast_r\Psi$ cannot be minimizing.
\end{proof}

\subsubsection{Equivalence of Banyaga's and Hofer's distances on $\ham{\momeg}$}\label{sec:equiv-bany-hofers-dist}

Banyaga proved that the (pseudo-)distance induced by $\ell^B$  is non-degenerate \cite[Theorem 1]{Banyaga07}. This immediately implies that $d_{(\mu^g,1)}$ is a genuine distance on $\sympo{\momeg}$.

Moreover, $d^B_{(\mu^g,1)}$ (its restriction to $\ham{\momeg}$) is bounded above by Hofer's distance (see Proposition \ref{prop:comparison-with-nb}) and Banyaga conjectured that these two distances are equivalent. We now prove this conjecture.

\begin{theo}\label{theo:equiv-bany-hofers}
  Let $n^B$ be Hofer's norm. The restriction of $d_{(\mu^g,1)}$ to $\ham{\momeg}$ is equivalent to Hofer's distance.
\end{theo}

The proof is based on Banyaga's proof of non-degeneracy, which goes as follows. First he shows that
\begin{align}\label{eq:B-thm}
    \begin{minipage}[]{0.9\linewidth}
      {if a symplectomorphism is at distance 0 from the identity, then \textit{it has to be Hamiltonian} and \textit{its Hofer norm is also 0}}
    \end{minipage}
\end{align}
and then he concludes by using the non-degeneracy of Hofer's distance. (Recall that we adapted the proof of \eqref{eq:B-thm} in order to prove Proposition \ref{prop:non-degeneracy}.)

In view of Proposition \ref{prop:comparison-with-nb}, to prove the equivalence we only need to prove the existence of a constant $D$, such that
$$\forall\, \phi \in\ham{\momeg},\qquad d_{\mathrm{Hofer}}(\id,\phi) \leq D\, d^B_{(\mu^g,1)}(\id,\phi)$$ 
or, equivalently, via the sequential criterion, it suffices to prove that
\begin{align}\label{eq:our-thm}
    \begin{minipage}[]{0.9\linewidth}
      {any sequence of Hamiltonian diffeomorphisms converging to the identity for $d_{(\mu^g,1)}$, converges to the identity for Hofer's distance.}
    \end{minipage}
\end{align}
(This leads to the equivalence of the induced topologies and thus of the distances.) By comparing \eqref{eq:B-thm} and \eqref{eq:our-thm}, it seems reasonable to try to adapt Banyaga's proof in order to obtain this, slightly, different result.

\begin{proof}[Proof of Theorem \ref{theo:equiv-bany-hofers}]
  Let $\{\phi^k\}_k$ be a sequence of Hamiltonian diffeomorphisms such that $d_{(\mu^g,1)}(\id,\phi^k)$ converges to 0 when $k$ goes to infinity. For any $k$ and any $\eps >0$, there exists a path of symplectomorphisms $\Phi^{k,\eps}=\{ \phi^{k,\eps}_t  \}_{t=0}^1$ such that
\begin{align}
  \label{eq:definition-of-Phikeps}
  \phi^{k,\eps}_0 = \id, \qquad \phi^{k,\eps}_1 = \phi^k \qquad \mbox{and} \qquad \ell^B(\Phi^{k,\eps}) < d_{(\mu^g,1)}(\id,\phi^k) + \eps \;.
\end{align}
By Hodge decomposition, these paths uniquely decompose as time-wise compositions
\begin{align*}
  \phi^{k,\eps}_t = \rho^{k,\eps}_t \mu^{k,\eps}_t \quad \mbox{with} \;\left\{
    \begin{array}[l]{l}
      \!\!\mu^{k,\eps} \mbox{ path of Hamiltonian diffeomorphisms, and}\\
      \!\!\rho^{k,\eps} \mbox{ path induced by purely harmonic $1$--forms.}
    \end{array}
  \right.
\end{align*}
Notice that for all $k$ and $\eps$, $\rho^{k,\eps}_1=\phi^{k,\eps}_1(\mu^{k,\eps}_1)^{-1}$ is Hamiltonian. 

We fix $\eps_0$. Since $d_{(\mu^g,1)}(\id,\phi^k)$ converges to 0, there exists $k_0$ such that for all $k\geq k_0$, $d_{(\mu^g,1)}(\id,\phi^k)<\eps_0/2$. This implies that for $\eps=\eps_0/2$ and all $k\geq k_0$, 
\begin{align}
  \label{eq:assumption-on-ellB}
  \ell^B(\Phi^{k,\eps}) = \int_0^1 \left\| \alpha(\rho^{k,\eps})_t\right\|_{L^2} dt + \ell^H(\mu^{k,\eps}) < d_{(\mu^g,1)}(\id,\phi^k) + \eps < \eps_0
\end{align}
(where $\ell^H$ denotes Hofer's length of paths of Hamiltonian diffeomorphisms). Thus, in particular, $d_{\mathrm{Hofer}}(\id,\mu^{k,\eps}_1)\leq \ell^H(\mu^{k,\eps})<\eps_0$ and $\flux(\rho^{k,\eps})$ admits a representative with $L^2$--norm at most $\eps_0$, since
\begin{align}  \label{eq:7}
 \left\| \int_0^1 \alpha_t(\rho^{k,\eps}) \,dt \right\|_{L^2} \leq \int_0^1 \left\| \alpha_t(\rho^{k,\eps})\right\|_{L^2} dt \leq \eps_0  \;.
\end{align}
Since $\rho^{k,\eps}_1$ is Hamiltonian, $\flux(\rho^{k,\eps})$ is in the Flux group which is discrete by Ono's theorem. Thus for $\eps_0$ small enough ($k_0$ chosen accordingly, $k\geq k_0$, and $\eps=\eps_0/2$), $\flux(\rho^{k,\eps})=0$. 

Now, all the data $\eps_0$, $\eps$, $k_0$, $k$ being fixed as specified above, we follow step by step Banyaga's proof of the non-degeneracy part of \cite[Theorem 1]{Banyaga07}, namely:
\begin{enumerate}
\item Since $\flux(\rho^{k,\eps})=0$, $\rho^{k,\eps}$ can be deformed to a path of Hamiltonian diffeomorphisms, $g^{k,\eps}$, connecting the identity to $\rho^{k,\eps}_1$ and generated by an explicit Hamiltonian (deformation which Banyaga introduced in \cite{Banyaga78}). 
\item The specific form of the generating Hamiltonian allows one to bound the Hofer length of $g^{k,\eps}$ in terms of $\eps_0$ and the norm of the Hamiltonian vector field inducing $g^{k,\eps}$.
\item Finally, the deformation process of step (1) is not too wild and one can bound the norm of the vector field inducing $g^{k,\eps}$ in terms of the norm of the vector field inducing  $\rho^{k,\eps}$. This, together with \eqref{eq:7} and step (2) give a bound on the Hofer length of $g^{k,\eps}$ only in terms of $\eps_0$.
\end{enumerate}
Thus step (3) gives a bound in $\eps_0$ on the Hofer distance between $\id$ and $\rho^{k,\eps}_1$. Finally, 
$$ d_{\mathrm{Hofer}}(\id,\phi^{k}) = d_{\mathrm{Hofer}}(\id,\rho^{k,\eps}_1 \mu^{k,\eps}_1) \leq d_{\mathrm{Hofer}}(\id,\rho^{k,\eps}_1) + d_{\mathrm{Hofer}}(\id,\mu^{k,\eps}_1) $$
and we just proved that the two terms on the right hand side converge to 0 when $k$ goes to infinity. The conclusion follows. (We refer to  \cite{Banyaga07} for the details of the steps (1)-(3) above.)
\end{proof}

\end{document}